\newcommand{\R}{\mathds R}

\newcommand{\Dds}{\tfrac{\mathrm D}{\mathrm ds}}

\documentclass[twoside,a4paper,11pt]{amsart}

\usepackage{amsmath}
\usepackage{times}

\usepackage{hyperref}
\usepackage{dsfont}
\usepackage{enumerate}
\usepackage{amssymb}
\usepackage{color}

\numberwithin{equation}{section}

\title[On the normal exponential map in singular conformal metrics]{On the normal exponential map in singular conformal metrics}
\author[R.\ Giamb\`o ,\ F.\ Giannoni]{Roberto Giamb\`o, Fabio Giannoni}
\address{Scuola di Scienze e Tecnologie \hfill\break\indent
Universit\`a di Camerino\hfill\break\indent Italy}
\email{roberto.giambo@unicam.it, fabio.giannoni@unicam.it}
\author[P. Piccione]{Paolo Piccione}
\address{Departamento de Matem\'atica\hfill\break\indent Instituto de
Matem\'atica e Estat\'\i stica
\hfill\break\indent Universidade de S\~ao Paulo
\hfill\break\indent Brazil}
\email{piccione@ime.usp.br}


\date{July 27th, 2014}

\begin{document}


\theoremstyle{plain}\newtheorem{teo}{Theorem}[section]
\theoremstyle{plain}\newtheorem{prop}[teo]{Proposition}
\theoremstyle{plain}\newtheorem{lem}[teo]{Lemma}
\theoremstyle{plain}\newtheorem{cor}[teo]{Corollary}
\theoremstyle{definition}\newtheorem{defin}[teo]{Definition}
\theoremstyle{remark}\newtheorem{rem}[teo]{Remark}
\theoremstyle{definition}\newtheorem{example}[teo]{Example}
\theoremstyle{remark}\newtheorem{step}{\bf Step}
\theoremstyle{plain}\newtheorem*{teon}{Theorem}
\theoremstyle{plain}\newtheorem*{conj}{Conjecture}
\theoremstyle{plain}\newtheorem*{defin*}{Definition}


\begin{abstract}
Brake orbits and homoclinics of autonomous dynamical systems correspond, via Maupertuis principle, to
geodesics in Riemannian manifolds endowed with a metric which is singular on the boundary (Jacobi metric).
Motivated by the classical, yet still intriguing in many aspects, problem of establishing multiplicity results for brake
orbits and homoclinics, as done in \cite{GGP1,esistenza, arma}, and by the development of a Morse theory in \cite{MT}
for  geodesics in such kind of metric, in this paper we study the related normal exponential map from a global perspective.
\end{abstract}

\maketitle
\renewcommand{\contentsline}[4]{\csname nuova#1\endcsname{#2}{#3}{#4}}
\newcommand{\nuovasection}[3]{\medskip\hbox to \hsize{\vbox{\advance\hsize by -1cm\baselineskip=12pt\parfillskip=0pt\leftskip=3.5cm\noindent\hskip -2cm #1\leaders\hbox{.}\hfil\hfil\par}$\,$#2\hfil}}
\newcommand{\nuovasubsection}[3]{\medskip\hbox to \hsize{\vbox{\advance\hsize by -1cm\baselineskip=12pt\parfillskip=0pt\leftskip=4cm\noindent\hskip -2cm #1\leaders\hbox{.}\hfil\hfil\par}$\,$#2\hfil}}

\section{Introduction}\label{sec:intro}
The purpose of this paper is to prove global regularity results for the distance-to-the-boundary function in Riemannian manifolds with singular metrics on the boundary. 
This kind of study is motivated by the use of the degenerate Jacobi metric (via Maupertuis' principle) for the problem of brake orbits and homoclinics in the autonomous case, as done in \cite{GGP1, esistenza,jde, arma}. This approach was suggested for the first time by Seifert in \cite{seifert}, where a famous conjecture concerning a multiplicity results for brake orbits was formulated.
The metric singularity on the boundary is of a very special type, being produced by the first order vanishing of a conformal factor
which multiplies a fixed background metric. Following the local theory developed in \cite{MT}, in this paper we will introduce a suitable notion of normal exponential map
adapted to this type of degenerate boundaries, and we will determine its regularity properties.

There exists a huge amount of literature concerning the study of brake orbits -- see e.g., \cite{liu2013,LZZ,zhang11,zhang13} -- and more generally on the study of periodic solutions of autonomous
Hamiltonian systems with prescribed energy \cite{liu2014,liu2002,Long,rab87,rab}.
We also observe here that  manifolds with singular boundary of the type investigated in the present paper 
arise naturally in the study of certain compactifications of incomplete Riemannian manifolds. They constitute an important class
of the so-called \emph{singular manifolds}, see \cite{amann13} and the references therein, where the singularity is described by the 
vanishing (or the diverging) of some conformal factor, called the \emph{singularity function}.

In order to describe the results of the present paper, let us consider a Riemannian manifold $(M,g)$
of class $C^3$,
representing the configuration space of some conservative dynamical system,
and let $V:M\to\mathds R$ be a map of class $C^2$ on $M$, which represents the potential function of the system.
Fix an energy level $E \in \mathds{R}$, $E>\inf\limits_MV$, and consider the Jacobi metric:
\begin{equation}\label{eq:Econf}
g_*=\tfrac12(E-V)g,
\end{equation}
defined in the open sublevel $V^{-1}\big(\left]-\infty,E\right[\big)$, the so called \emph{potential well}.
Note that $g_*$ is singular  on the boundary $V^{-1}(E)$.

For any $Q \in V^{-1}\big(\left]-\infty,E\right[\big)$, denote by $d_V(Q)$ the distance of $Q$ from $V^{-1}(E)$ with respect to the Jacobi metric \eqref{eq:Econf}.
In the recent work \cite{MT},  the following assumptions:
\begin{itemize}\label{eq:Jacobi}
\item $V$ is of class $C^2$ in a neighborhood of $V^{-1}\left(\left]-\infty,E\right[\right)$;
\item $E$ is a regular value for $V$;
\item the sublevel $V^{-1}\left(\left]-\infty,E\right]\right)$ is compact;
\end{itemize}
were used to prove that if the minimizer that realizes $d_V(Q)$ is unique, then $d_V$ is differentiable at $Q$,
and its gradient with respect to the Riemann metric $g$ is given by
\[\nabla^gd_V(Q)=\frac{E-V(Q)}{2d_V(Q)}\dot \gamma_Q(1),\] where $\gamma_Q$ is the minimizer (affinely parametrized in the interval $[0,1]$) joining $V^{-1}(E)$ with $Q$.
Uniqueness of the minimizer is guaranteed for all points $Q$ sufficiently close to the boundary $V^{-1}(E)$.
Moreover, in \cite[Section~4]{MT} a definition of Jacobi fields along Jacobi geodesic starting from $V^{-1}(E)$ is given and a Morse Index Theorem was proved.

Following along this path, in this paper we introduce e normal exponential map $\exp^\perp$, defined in terms of
$g_*$-geodesics $\gamma:\left]0,a\right]\to V^{-1}\big(\left]-\infty,E\right[\big)$ satisfying
$\lim\limits_{s\downarrow0}\gamma=P\in V^{-1}(E)$, see Section~\ref{sec:esp}. Such geodesic is necessarily ``orthogonal''
to $V^{-1}(E)$, in the sense that a suitable normalization of $\dot\gamma(s)$, when $s$ goes to $0$, admits as limit as a vector $v\in T_P\big(V^{-1}(E)\big)$ which is $g$-orthogonal to $V^{-1}(E)$ at $P$.
We prove the regularity of $\exp^\perp$, and we establish the equivalence between conjugate points to $V^{-1}(E)$ and critical values of the exponential map (Proposition~\ref{prop:expC1} and Theorem~\ref{thm:equiv}).

In section \ref{sec:reg} we apply the above result to prove that if the minimizer between $V^{-1}(E)$ and $Q_0$ is unique and if $Q_0$ is not conjugate to $V^{-1}(E)$, then $d_V(Q)$ is of class $C^2$ in a neighborhood of $Q_0$ (cf.\ Theorem \ref{teo:C2reg}). This extends
the result proved in \cite{GGP1}, where the $C^2$--regularity is proved only for points $Q_0$ sufficiently close to $V^{-1}(E)$.

\section{Exponential map and focal points}\label{sec:esp}
A geodesic $x:I\subset\mathds R\to V^{-1}\big(\left]-\infty,E\right[\big)$ relative to the metric $g_*$ \eqref{eq:Econf} will be called a \emph{Jacobi geodesic};
such a curve satisfies the second order differential equation:
\begin{equation}\label{eq:equazioneJacgeod}
\big(E-V(x(s))\big)\Dds\dot x(s) - g\big(\nabla V(x(s)),\dot
x(s)\big)\dot x(s) +\\
\frac 12 g\big(\dot x(s),\dot x(s)\big)\nabla
V(x(s))=0.
\end{equation}
Moreover, a non costant Jacobi geodesic $\gamma$ satisfies the conservation law
\begin{equation}\label{eq:conslaw}
\tfrac12\big(E-V(\gamma)\big)g(\dot\gamma,\dot\gamma)\equiv \lambda_\gamma\in \mathds{R}^{+}\setminus\{0\}.
\end{equation}
Given a Jacobi geodesic $\gamma:\left]0,a\right]\to V^{-1}\big(\left]-\infty,E\right[\big)$
satisfying $\lim\limits_{s\downarrow0}\gamma(s)=P\in V^{-1}(E)$, then the $g$-normalized
tangent vector $v_s:=\dot\gamma(s)/g\big(\dot\gamma(s),\dot\gamma(s)\big)^\frac12$ admits limit
$\lim\limits_{s\downarrow0}v_s=v_0\in T_P\big(V^{-1}(E)\big)^\perp$.
Indeed, using \cite[Lemma 2.2]{MT}, it can be seen that $v_s+\nabla V(\gamma(s))/g\left(\nabla V(\gamma(s)),\nabla V(\gamma(s))\right)^{1/2}$ is bounded in norm by an  function that is infinitesimal for $s\to 0$.
In this situation, we will say that $\gamma$ is a Jacobi geodesic \textit{starting orthogonally} to $V^{-1}(E)$. These geodesics will be used later in the definition of the normal exponential map of $V^{-1}(E)$ using these geodesics.

\begin{rem}\label{rem:diffeo}
Let us now recall a basic result that will be repeatedly used throughout the paper. We refer the reader to \cite[eq (3.16) and following discussion]{MT} for the details of this construction, that we will now briefly sketch.
For any $P\in V^{-1}(E)$ consider the trajectories $t\mapsto q(t,P)$ that are solutions of the Cauchy problem:
\begin{equation}\label{eq:startingesterne}
\left\{
\begin{aligned}
& \tfrac{\mathrm D}{\mathrm dt}\dot q  + \text{grad }V(q) = 0 \\
& q(0) = P \\
& \dot q(0)=0,
\end{aligned}
\right.
\end{equation}
where $\frac{\mathrm D}{\mathrm dt}$ is the covariant derivative of vector fields along $q$, and $\text{grad }V$ is the gradient of $V$ with respect to the Riemannian metric $g$.

Moreover, called $\gamma(P,\lambda) $  the unique Jacobi geodesic starting from $P \in V^{-1}(E)$ and satisfying
$\lambda_\gamma=\lambda$ (note that the uniqueness of $\gamma(P,\lambda)$  follows again from the Maupertuis Principle),
we have that
$$ q(t,P)=\gamma(P,\lambda)(s),\quad\text{where\ }
t= t(s)=
\int_0^s \frac{\sqrt{\lambda}\,\mathrm dr}{E-V\big(\gamma(P,\lambda)(r)\big)},$$
and then
 $t\mapsto q(t,P)$ is a reparameterization of $s\mapsto\gamma(P,\lambda)(s)$.

Using Maupertuis--Jacobi principle \cite[Proposition 2.1]{MT}, one can show that, setting $\tau=\sqrt t$, the map $q(\tau,P)$,  defined in $[0,\tau_0[\times V^{-1}(E)$, for a suitable $\tau_0$ sufficiently small, is a  $C^1$--diffeomorphism.
Moreover, $q(\tau,P)$ is a coordinate system in a neighborhood of $V^{-1}(E)$ and  $\frac{\partial q}{\partial \tau}(0,P)=-\frac12\text{grad }V(P)$, while $\frac{\partial q}{\partial P}(0,P)$ is the identity map.
In particular, it is worth remarking that a $C^1$-diffemorphism, say $\Phi$, between $V^{-1}(E)$ and $d_V^{-1}(\delta)$ is obtained by setting $\Phi(P)=q(\delta,P)$.
\end{rem}

\begin{defin}\label{def:exp}
For any $s > 0, \lambda > 0$ and $P \in V^{-1}(E)$ we denote by \textit{exponential map} (starting from $V^{-1}(E)$) the map:
\begin{equation}\label{eq:defesp}
\exp^{\perp}(P,\lambda)(s) = \gamma(P,\lambda)(s),
\end{equation}

\end{defin}

First of all, let us prove the following
\begin{prop}\label{prop:expC1} Fix any $\delta> 0$ such that $d_V^{-1}(\delta)$ is a $C^2$--hypersurface (cf \cite{GGP1,MT}).
Let $\Phi$ be the $C^1$-diffemorphism between $V^{-1}(E)$ and $d_V^{-1}(\delta)$ (see Remark \ref{rem:diffeo}). Let $\gamma_{\delta}$ be the solutions of the Cauchy problem for Jacobi geodesics with initial position,  $\gamma(\frac{\delta}{\sqrt{\lambda}})$ and initial speed $\dot \gamma(\frac{\delta}{\sqrt{\lambda}})$. Then
for any $s > 0$
\begin{equation}\label{eq:trastemp}
\gamma(P,\lambda)(s) = \gamma_{\delta}(P,\lambda)\big(s-\tfrac{\delta}{\sqrt{\lambda}}\big),
\end{equation}
and $\exp^{\perp}(\cdot,\cdot)(s)$ is of class $C^1$.
\end{prop}
\begin{proof}

First of all, recall from Remark \ref{rem:diffeo} that
$q(\sqrt{t},P)$ is a coordinate system in a neighborhood of $V^{-1}(E)$.
Therefore, for any $\delta > 0$ such that $\gamma(P,\lambda)$  is defined in a neighborhood of $\frac{\delta}{\sqrt{\lambda}}$ (the value of the parameter at which $\gamma$ reaches $d_V^{-1}(\delta)$), there exists one and only one $t_\delta(P,\lambda)$ such that $q(t_\delta(P,\lambda),P)$ intersects $d_V^{-1}(\delta)$.

Using the Implicit Function Theorem, we deduce the existence of $\delta$ sufficiently small such that the map $t_\delta(P,\lambda)$ is of class $C^1$.

Now, since
\[
\dot \gamma(P,\lambda)(s)=\frac{dt}{ds} \dot q(t,P)\frac{dt}{ds}= \frac{\sqrt{\lambda}}{E-V(q(t,P))}\dot q(t,P)
\]
we have that \eqref{eq:trastemp} holds for any $s > 0$. Now recall that
$\gamma_{\delta}$ is the Jacobi geodesic satisfying the initial conditions
\begin{equation}\label{eq:datiiniziali}
\left\{
\begin{aligned}
& \gamma_\delta(P,\lambda)(0) = q(t_\delta(P,\lambda),P), \\
& \dot \gamma_{\delta}(P,\lambda)(0)=  \frac{\sqrt{\lambda}}{E-V(q(t_\delta(P,\lambda)),P)}\dot q(t_\delta(P,\lambda),P),
\end{aligned}
\right.
\end{equation}

Since $(t,P)\mapsto q(t,P)$ and $(t,P)\mapsto\dot q(t,P)$ are of class $C^1$ (Remark \ref{rem:diffeo}), standard regularity properties of ODE's  give the $C^1$--regularity of
the map $(P,\lambda)\mapsto\dot \gamma_{\delta}(P,\lambda)\big(s-\frac{\delta}{\sqrt{\lambda}}\big)$  concluding the proof.
\end{proof}

\begin{rem}\label{eq:regvel}
Note that classical regularization methods show that  the map $(P,\lambda)\mapsto\dot \gamma_{\delta}(P,\lambda)\big(s-\frac{\delta}{\sqrt{\lambda}}\big)$ is also of class $C^1$.
\end{rem}

\begin{rem}\label{rem:ort}
It is important to note that $\gamma_\delta$ starts from the point:
\[P_\delta := q\big(t_\delta(P),P\big) \in d_V^{-1}(\delta) =: N_\delta,\]
with velocity orthogonal to the hypersurface $N_\delta$.
\end{rem}

\begin{rem}\label{rem:hessconf}
Consider a Riemannian metric conformal to $g$, say $\phi g$ with $\phi$ positive and smooth real map, and the corresponding action integral
\[
h(x)= \frac12\int_0^1\phi(x)g(\dot x, \dot x)\,\mathrm ds
\]
on the space $X$ of the $H^{1,2}$--curves from $[0,1]$ to $M$ such that $x(0) \in N$, $x(1)=Q$, where $N$ is a smooth hypersurface of $M$, and $Q$ is a fixed point in $M$. We recall that the critical points $\gamma$ on $h_{|X}$ satisfy the ordinary differential equation:
\[
\frac{\mathrm D}{\mathrm ds}\big[\varphi\big(\gamma(s)\big)\dot\gamma(s)\big]=\tfrac12 g\big(\dot\gamma(s),\dot\gamma(s)\big)\nabla\varphi(\gamma(s)),
\]
and the boundary conditions
\[
\gamma(0) \in N, \; \gamma(1)=Q.
\]
(Here $\nabla\varphi$ denotes the gradient of $\varphi$ with respect to the metric $g$).
Moreover, the tangent space of $X$ at $\gamma$ is given by the $H^{1,2}$ vector fields $\xi$ along $\gamma$ such that
$\xi(0) \in T_{\gamma(0)}N$, the tangent space of $N$ at $\gamma(0)$ and $\xi(1)=0$,
while the Hessian of $h$ at a critical point $\gamma$ is given by
\begin{multline*}
H^{h}(\gamma)=\int_0^1\frac12 g(\dot\gamma,\dot\gamma)g(H^\varphi(\gamma)[\xi],\xi)+2g(\nabla\varphi(\gamma),\xi)g(\tfrac{\mathrm D}{\mathrm ds}\xi,\dot\gamma)\\+\varphi(\gamma)\left[g(R(\xi,\dot\gamma)\xi,\dot\gamma)+g(\tfrac{\mathrm D}{\mathrm ds}\xi,\tfrac{\mathrm D}{\mathrm ds}\xi)\right]\,\mathrm ds-\varphi(\gamma)g(\nabla_\xi\xi,\dot\gamma)\vert_{s=0}
\end{multline*}
(where $\nabla$ is the covariant derivative with respect to the metric $g$).
By polarization, we obtain the related bilinear form:
\begin{multline}\label{eq:index-form}
I(\xi,\eta)=\int_0^1\Big[\frac12 g(\dot\gamma,\dot\gamma)g(H^\varphi(\gamma)[\xi],\eta)\\
+g(\nabla\varphi(\gamma),\xi)g(\tfrac{\mathrm D}{\mathrm ds}\eta,\dot\gamma)+g(\nabla\varphi(\gamma),\eta)g(\tfrac{\mathrm D}{\mathrm ds}\xi,\dot\gamma)
\\+\varphi(\gamma)\left[g(R(\dot\gamma,\xi)\dot\gamma,\eta)+g(\tfrac{\mathrm D}{\mathrm ds}\xi,\tfrac{\mathrm D}{\mathrm ds}\eta)\right]\,\Big]\mathrm ds+\\
-\frac12\varphi(\gamma)\left[g(\nabla_\xi\eta,\dot\gamma)+g(\nabla_\eta\xi,\dot\gamma)\right]\Big\vert_{s=0}.
\end{multline}
The Jacobi field equation is:
\begin{multline*}
\frac12g(\dot\gamma,\dot\gamma) H^\varphi(\gamma)[\xi]-\tfrac{\mathrm D}{\mathrm ds}\left[g\big(\nabla\varphi(\gamma),\xi\big)\dot\gamma+
\varphi(\gamma)\tfrac{\mathrm D}{\mathrm ds}\xi\right] \\
+ g\big(\tfrac{\mathrm D}{\mathrm ds}\xi,\dot\gamma\big)\nabla\varphi(\gamma)+\varphi(\gamma) R(\dot\gamma,\xi)\dot\gamma=0,
\end{multline*}
while the vector fields in the kernel of $I$ are Jacobi fields $\xi$ such that $\xi(1)=0$ and
\begin{equation}\label{eq:intparti}
\varphi(\gamma)\tfrac{\mathrm D}{\mathrm ds}\xi+g\big(\nabla\varphi(\gamma),\xi\big)\dot\gamma-\varphi(\gamma)\nabla_\xi\dot\gamma
\end{equation}
is orthogonal to $T_{\gamma(0)}N$ at the instant $s=0$.

As pointed out in \cite[Definition 4.6]{MT},
for the Jacobi metric the Jacobi fields differential equation in the interval $[0,a]$ is
\begin{multline}\label{eq:eqkernel}
-\frac{\mathrm D}{\mathrm ds}\Big((E-V(\gamma))\tfrac{\mathrm D}{\mathrm ds}\xi\Big) + \big(E-V(\gamma)\big)R(\dot \gamma,\xi)\dot \gamma+\\
+\frac{\mathrm D}{\mathrm ds}\big(g(\nabla V(\gamma),\xi)\dot \gamma\big)- g(\dot \gamma,\tfrac{\mathrm D}{\mathrm ds}\xi)\nabla V(\gamma)+\\
-\frac12g(\dot\gamma,\dot\gamma)H^{V}(\gamma)[\xi]=0 \text{ for all }s \in\left]0,a\right].
\end{multline}
\end{rem}

We recall now the definition (given in \cite{MT}) of point which is conjugate to $V^{-1}(E)$.
\begin{defin}\label{def:punto-coniugato}
A point $\gamma(s_0)$, on a Jacobi geodesic $\gamma$ starting from $V^{-1}(E)$ is said to be \textit{conjugate} to $V^{-1}(E)$ if there exists a non
identically zero vector field $\xi$ along $\gamma$ defined in $[0,a]$, with $\xi(a)=0$, such that
\begin{enumerate}[(a)]
\item\label{itm:a} $\xi \in C^0\big([0,s_0]\big) \cap C^2\big(]0,a]\big)$;\smallskip

\item$\int_0^{a}(E-V(\gamma))g\big(\tfrac{\mathrm D}{\mathrm ds}\xi,\tfrac{\mathrm D}{\mathrm ds}\xi\big)\,\mathrm ds < +\infty$;
\smallskip

\item $\xi$ satisfies equation \eqref{eq:eqkernel} in $]0,a]$;
\smallskip

\item\label{itm:d} $\xi(0)\in T_{\gamma(0)}V^{-1}(E)$;\smallskip

\item\label{itm:e} the continuous extension at $s=0$ of the vector field:
\begin{equation}\label{eq:contvecfield}
\big(E-V(\gamma)\big)\tfrac{\mathrm D}{\mathrm ds}\xi - g\big(\nabla V(\gamma),\xi\big)\dot \gamma
\end{equation}
is a multiple of $\nabla V(\gamma(0))$.
\end{enumerate}
Recall that $\nabla V(\gamma(0))$ is parallel to the limit unit  vector of $\dot \gamma$.
A vector field $\xi$ along $\gamma$ satisfying (a)---(e) above will be called an \emph{$E$-Jacobi field}.
\end{defin}


Let $\gamma:[0,a]\to M$ be a Jacobi geodesic starting orthogonally to $V^{-1}(E)$. For $s>0$ small
enough, let $\mathcal N_s$ denote the set of points having $d_V$-distance from $V^{-1}(E)$ equal to $s$, so that
$p_s=\gamma(\sigma_s)\in\mathcal N_s$, where $\sigma_s=s/\sqrt{\lambda_\gamma}$.
Observe that, if $s>0$ is sufficiently small, $\mathcal{N}_s$ is a $C^2$ embedded hypersurface of $M$, as shown in \cite{GGP1}.
Denote by $\Sigma_s^*$ the shape operator of ${\mathcal N}_s$ at
$p_s$ relatively to the metric $g_*$.
From the relation
\begin{equation}\label{eq:nabla*}
\nabla^*_X Y=\nabla_X Y-\frac1{2(E-V)}\left[g(\nabla V, X)Y+g(\nabla V,Y)X-g(X,Y)\nabla V\right]
\end{equation}
the following expression for the covariant differentiation $\frac{\mathrm D^*}{\mathrm ds}$ along $\gamma$ relative to the
Levi--Civita connection of $g_*$  holds:
\begin{equation}\label{eq:Ds*}
\frac{\mathrm D^*}{\mathrm ds}\eta=
\frac{\mathrm D}{\mathrm ds}\eta-\frac{1}{2(E-V(\gamma))}\left[g(\nabla V(\gamma),\dot\gamma)\eta+g(\nabla V(\gamma),\eta)\dot\gamma-g(\dot\gamma,\eta)\nabla V(\gamma)\right].
\end{equation}
Using \eqref{eq:nabla*} again, given $\eta,\xi\in T_{p_s}{\mathcal N}_s$,  the shape operator with respect to the conformal metric $\Sigma^*_s$  satisfies the identity
\begin{equation}\label{eq:shape*}
g_*\big(\Sigma^*_s(\xi),\eta\big)=\frac12(E-V(\gamma))g(\Sigma_s(\xi),\eta)+\frac14g(\xi,\eta)g(\nabla V(\gamma),\dot\gamma),
\end{equation}
where we have exploited the fact that $\dot\gamma\in(T_{p_s}{\mathcal N}_s)^\perp$.

Moreover, introduced the Riemann curvature tensor of $g_*$, i.e. $$R^*(X,Y)=[\nabla^*_X,\nabla^*_Y]-\nabla^*_{[X,Y]},$$
then it can be seen that \eqref{eq:eqkernel} is equivalent to the equation
\begin{equation}\label{eq:Jacobieq*}
\left(\tfrac{\mathrm D^*}{\mathrm ds}\right)^2\xi(s)=R^*\big(\dot\gamma(s),\xi(s)\big)\dot\gamma(s).
\end{equation}
This equation and the  skew--symmetry  of the Riemann tensor $R^*$ also imply that,
for every $E$-Jacobi field along $\gamma$, the  quantity $\lambda_\xi=2\,g_*\big(\tfrac{\mathrm D^*}{\mathrm ds}\xi(s),\dot\gamma(s)\big)$ is constant
on $]0,a]$. Also observe that  \eqref{eq:Ds*} implies that:
\begin{equation}\label{eq:lambdaxi}
\lambda_\xi=2\,g_*\big(\tfrac{\mathrm D^*}{\mathrm ds}\xi,\dot\gamma\big)=-\tfrac12g\big(\nabla V(\gamma),\xi)\,g(\dot\gamma,\dot\gamma)+\big[E-V(\gamma)\big]\,g\big(\tfrac{\mathrm D}{\mathrm ds}\xi,\dot\gamma).
\end{equation}

\begin{rem}\label{rem:cost}
Observe that, if $\xi$ is an $E$--Jacobi field along $\gamma$ with $\xi(a)=0$, then $g_*(\xi,\gamma)$ identically vanishes on $[0,a]$.
Indeed, we have  just proved that the map
$g_*\big(\tfrac{\mathrm D^*}{\mathrm ds}\xi(s),\dot\gamma(s)\big)$ is constant, from which it follows that  $g_*(\xi,\gamma)$ is an affine function.  Moreover, $\xi(a)=0$ implies that $g_*(\xi,\gamma)$ vanishes at $s=a$. Let us now prove that $g_*(\xi,\gamma)$ can be extended continuously by setting it equal to $0$ at $s=0$.

First, observe that $\xi$ is continuous at $s=0$  (by condition \eqref{itm:a} of Definition~\ref{def:punto-coniugato}), and  it can be seen that $(E-V(\gamma))$ behaves like $s^{2/3}$ near $s=0$,
see \cite[Remark 2.4]{MT}. The conservation law \eqref{eq:conslaw}  then implies that $\dot\gamma$ behaves like $s^{-1/3}$ near $s=0$, and therefore
$g_*(\xi,\gamma)$ vanishes at $s=0$. Thus, $g_*(\xi,\gamma)$  vanishes identically  on $[0,a]$ and, in particular, $\xi(\sigma_s)\in T_{\gamma(\sigma_s)}{\mathcal N}_s$.
\end{rem}
Finally, we can give the main result of this section:
\begin{teo}\label{thm:equiv}
The point $\gamma(a)$ is conjugate to $V^{-1}(E)$ if and only if $\gamma(a)$ is a critical value of $\exp^{\perp}(\cdot,\cdot)(a)$.
\end{teo}
The proof is divided into two Propositions.
\begin{prop}\label{thm:conjcritical}
Let $\gamma(a)$ be conjugate to $V^{-1}(E)$. Then it is a critical value of $\exp^{\perp}(\cdot,\cdot)(a)$.
\end{prop}
\begin{proof}
Let $\xi$ be an $E$--Jacobi field in the sense of Definition \ref{def:punto-coniugato} such that $\xi(a)=0$. Fix $\delta > 0$ sufficiently small and consider $\mathcal N_{\frac{\delta}{\sqrt{\lambda}}}$ and $\xi({\frac{\delta}{\sqrt{\lambda}}})$. By Remarks~\ref{rem:ort} and \ref{rem:cost}, we have that $\xi({\frac{\delta}{\sqrt{\lambda}}}) \in T_{\gamma({\frac{\delta}{\sqrt{\lambda}}})}\mathcal N_{\frac{\delta}{\sqrt{\lambda}}}$. Now
fix $Q(r)$ a smooth curve in $\mathcal N_{\frac{\delta}{\sqrt{\lambda}}}$ such that $Q'(0)=\xi({\frac{\delta}{\sqrt{\lambda}}})$. If we fix $\delta$ and $\lambda$, the map
\[
P \in V^{-1}(E) \rightarrow q(P,t_\delta(P,\lambda))
\]
is a diffeomorphism $\Phi(P)$ of class $C^1$. Choose $P(r)=\Phi^{-1}(Q(r))$ and keep $\lambda$ fixed. Note that by \eqref{eq:defesp} with $s=a$
\[
\mathrm d\exp^{\perp}(P,\lambda)(a)[P'(0)] = \mathrm d\gamma(P,\lambda)(a)[P'(0)]= \lim_{r\to 0}\frac{\mathrm d}{\mathrm dr}\gamma(P(r),\lambda).
\]
Then, by \eqref{eq:trastemp} and standard argument in the classical theory of ODE's, the values of $d\gamma(P,\lambda)(a)[P'(0)]$ is given by the solution $z$ of the linearized equation of geodesics, evaluated at $s=a$ with initial position $\xi\left({\frac{\delta}{\sqrt{\lambda}}}\right)$ and initial velocity $\frac{D}{ds}\xi\left({\frac{\delta}{\sqrt{\lambda}}}\right)$. Then by the uniqueness of the solutions of the Cauchy problem we have $z(s)=\xi(s)$ for any $s$, so
\[
\mathrm d\exp^{\perp}(P,\lambda)(a)[P'(0)]=\xi(a)=0,
\]
concluding the proof.
\end{proof}
Before proving the converse, we need the following:
\begin{lem}\label{lem:costanza}
Let $\xi$ be an $\mathcal N_{\delta}$-Jacobi field along $\gamma$. Then for all $s \in\left]0,\delta\right]$ we have
$\xi(\sigma_s)\in T_{\gamma(\sigma_s)}{\mathcal N}_s$, and
the vector
\begin{equation}\label{eq:der+sigmaxi}
\left.\tfrac{\mathrm D^*}{\mathrm d\sigma}\right\vert_{\sigma=\sigma_s}\xi(\sigma_s)
+\Sigma_s^*\big(\xi(\sigma_s)\big)
\end{equation}
is parallel to $\dot\gamma(\sigma_s)$, where  $\sigma_s=s/\sqrt{\lambda_\gamma}$.
\end{lem}
\begin{proof}
Let us fix $s_0 \in\left]0,\delta\right]$ such that $\gamma(\sigma_{s_0})\in\mathcal{N}_{s_0}$, and prove that \eqref{eq:der+sigmaxi}
evaluated when $s=s_0$ is parallel to $\dot\gamma(\sigma_{s_0})$. To this aim, choose
 an arbitrary $\bar s\in]0,s_0[$ and let $r\mapsto\gamma_r$ be a $1$-parameter family of $g_*$-geodesics
$\gamma_r:[\sigma_{\bar s},\sigma_{s_0}]\to M$, $r\in\left]-\varepsilon,\varepsilon\right[$, such that for all $s\in[{\bar s},{s_0}]$,

\begin{itemize}
\item $\frac{\mathrm d}{\mathrm dr}\big\vert_{r=0}\gamma_r(\sigma_s)=\xi(\sigma_s)$,
\smallskip
\item $\gamma_r(\sigma_s)\in{\mathcal N}_s$.
\end{itemize}
Now let us fix $v\in\dot\gamma(\sigma_{s_0})^\perp$,  consider the two-parameter map $z(r,\sigma)=\gamma_r(\sigma)$, and let $\nu(r,\sigma)$ be a smooth vector field
along $z$ such that $\nu(0,\sigma_{s_0})=v$ and
\begin{equation}\label{eq:nuparallelo}
\tfrac{\mathrm D^*}{\mathrm d\sigma}\nu(r,\sigma)=0
\end{equation}
for all $r$.
Then using the properties of the family $\gamma_r$ we get
\begin{equation}\label{eq:primaform}
g_*\Big(\Sigma^*_{s_0}\big(\xi(\sigma_{s_0})\big),v\Big)
=g_*\big(\Sigma_{s_0}^*\big(\xi(\sigma_{s_0})\big), \nu(0,\sigma_{s_0})\big)=g_*\big(\tfrac{\mathrm D^*}{\mathrm dr}\big\vert_{r=0}\nu(r,\sigma_{s_0}),\dot\gamma(\sigma_{s_0})\big),
\end{equation}
where last equality is obtained easily using the properties of the shape operator $\Sigma_s^*$.

Now let us prove that the function $\sigma\mapsto g_*\big(\tfrac{\mathrm D^*}{\mathrm d\sigma}\xi(\sigma)+\Sigma_{s}^*\big(\xi(\sigma)\big), \nu(0,\sigma)\big)$ is constant (we drop the subscript $s$ from $\sigma$ and recall that $s$ depends on $\sigma$, i.e. $s=\sigma\sqrt{\lambda_\gamma}$).
From \eqref{eq:Jacobieq*} and \eqref{eq:nuparallelo} we obtain:
\begin{multline}\label{eq:derivataxi}
\frac{\mathrm d}{\mathrm d\sigma}\,g_*\big(\tfrac{\mathrm D^*}{\mathrm d\sigma}\xi(\sigma), \nu(0,\sigma)\big)\\ =g_*\big(\tfrac{(\mathrm D^*)^2}{\mathrm d\sigma^2}\xi(\sigma),\nu(0,\sigma)\big)=g_*\big(R^*\big(\dot\gamma(\sigma),\xi(\sigma)\big)\dot\gamma(\sigma),\nu(0,\sigma)\big).
\end{multline}
Using \eqref{eq:nuparallelo} and \eqref{eq:primaform}, we get:
\begin{multline}\label{eq:derivatasigma}
\frac{\mathrm d}{\mathrm d\sigma}\,g_*\big(\Sigma^*_s\big(\xi(\sigma)\big), \nu(0,\sigma)\big)=g_*\big(\tfrac{\mathrm D^*}{\mathrm d\sigma}\tfrac{\mathrm D^*}{\mathrm dr}\big\vert_{r=0}\nu(r,\sigma),\dot\gamma(\sigma)\big)\\
=g_*\big(\tfrac{\mathrm D^*}{\mathrm dr}\big\vert_{r=0}\tfrac{\mathrm D^*}{\mathrm d\sigma}\nu(r,\sigma),\dot\gamma(\sigma)\big)+g_*\big(R^*\big(\dot\gamma(\sigma),\xi(\sigma)\big)\nu(0,\sigma),\dot\gamma(\sigma)\big)\\=-g_*\big(R^*\big(\dot\gamma(\sigma),\xi(\sigma)\big)\dot\gamma(\sigma),\nu(0,\sigma)\big).
\end{multline}
Finally, using \eqref{eq:derivataxi} and \eqref{eq:derivatasigma} we obtain:
\[\frac{\mathrm d}{\mathrm d\sigma}\,g_*\big(\tfrac{\mathrm D^*}{\mathrm d\sigma}\xi(\sigma)+\Sigma_{s}^*\big(\xi(\sigma)\big), \nu(0,\sigma)\big)=0,\]
i.e., $g_*\big(\tfrac{\mathrm D^*}{\mathrm d\sigma}\xi(\sigma)+\Sigma_{s}^*\big(\xi(\sigma)\big), \nu(0,\sigma)\big)$ is constant for $\sigma\in[\sigma_{\bar s},\sigma_{s_0}]$.
But if we consider $s_0=\delta$, by assumptions of $\xi$ we have
\[
g_*\big(\tfrac{\mathrm D^*}{\mathrm d\sigma}\xi(\sigma)+\Sigma_{s}^*\big(\xi(\sigma)\big), \nu(0,\sigma)\big)\vert_{\sigma=\sigma_{s_0}}=0
\]
and the proof is complete.
\end{proof}

\begin{prop}
Suppose that $\gamma(a)$ is a critical value of the exponential map $\exp^{\perp}(\cdot,\cdot)(a)$. Then $\gamma(a)$ is conjugate to $V^{-1}(E)$.
\end{prop}
\begin{proof}
Let $\gamma(a)$ be a critical value  for the exponential map. Then  by \eqref{eq:trastemp} it is a critical value for the exponential map
defined by geodesics starting orthogonally from ${\mathcal N}_{\delta}$.

Then by classical results (cf \cite{docarmo}, Proposition 4.4, cap.10), there exists a $C^2$ Jacobi field  along $\gamma$ defined
in the interval $[\sigma_{\delta}\equiv\frac{\delta}{\sqrt{\lambda}},a]$ such that
\begin{equation}\label{eq:campoalbordo}
\xi(\sigma_\delta) \in T_{\gamma(\sigma_\delta)}{\mathcal N}_{\delta}, \quad \xi(a)=0
\end{equation}
and
\begin{equation}\label{eq:shapebordo}
\tfrac{\mathrm D}{\mathrm ds}\xi(\sigma_\delta) + \Sigma_{\delta}(\xi(\sigma_\delta)) \text{ is parallel to }\dot \gamma(\sigma_\delta).
\end{equation}
Note that $\xi$ can be extended as Jacobi field along $\gamma$ to the whole interval $\left]0,a\right]$. Moreover,
since $g_*(\xi(s),\dot \gamma(s))$ has second derivative identically zero, and it is null at $\sigma_\delta$ and $a$, we deduce that
\begin{equation}\label{eq:i}
g_*(\xi(s),\dot \gamma(s))=0 \text{ for any }s \in\left]0,a\right].
\end{equation}
Moreover, by Lemma \ref{lem:costanza}:
\begin{equation}\label{eq:ii}
\tfrac{\mathrm D^*}{\mathrm ds}\xi(\sigma) + \Sigma_{\delta}(\xi(\sigma)) \text{ is parallel to }\dot \gamma(\sigma)
 \text{ for any }
\sigma\in\left]0,\sigma_\delta\right].
\end{equation}
Now denote by $I_s^a(\xi,\eta)$ the index form \eqref{eq:index-form} with the interval $[0,1]$ replaced by the interval $[s,a])$ ($s > 0$) and $\phi$ replaced by $E-V$. By \eqref{eq:i} and \eqref{eq:ii} we deduce that
\begin{multline}\label{eq:eqdebole}
I_s^a(\xi,\eta)=0 \text{ for any smooth vector field $\eta$ along $\gamma$ satisfying}\\
\eta(s) \in T_{\gamma(s)}{\mathcal N}_{s\sqrt{\lambda}}, \;\text{and}\; \eta(a)=0.
\end{multline}
Now, arguing as in the proof of \cite[Propositions 4.15]{MT}, we show that there exists $s_{**} < s_* \in\left]0,\sigma_\delta\right]$ such that
for any $s \in\left]0,{s_{**}}\right[$ the quadratic form $I_s^{s_*}(\eta,\eta)$ has a minimizer in the affine space
$Y_s^{s_*}$ of the absolutely continuous vector fields $\eta$ along $\gamma$ such that
$\eta(s) \in T_{\gamma(s)}\mathcal N_{s\sqrt{\lambda}}$ and $\eta(s_*)=\xi(s_*)$. Such a minimizer is a Jacobi field. Moreover, as the proof of  \cite[Proposition 4.16]{MT}, we see also that $I_s^{s_*}$ is
strictly positive definite and there is only one unique Jacobi field along $\gamma$ in $Y_s^{s_*}$: therefore it coincides with the Jacobi field $\xi$.

The same estimate in the proof of  \cite[Proposition 4.15]{MT} gives also the existence of a constant $C$ independent on $s$ such that
\[
\int_s^{s_*}\big((E-V(\gamma)\big)g\big(\tfrac{\mathrm D}{\mathrm ds}\xi,\tfrac{\mathrm D}{\mathrm ds}\xi\big)\,\mathrm ds \leq C,
\]
from which we deduce (d) of Definition \ref{def:punto-coniugato}. Note that property (b) implies the continuity of $\xi$ in a $s=0$, and taking the limit as $s\to0$ in \eqref{eq:i} (using the unit  vector of $\dot \gamma$) gives also property (d).

Finally, integrating by parts in \eqref{eq:eqdebole} (with $s=0$) and using standard regularization methods as in \cite{MT} we obtain the continuity at $s=0$ of the map
\[\big(E-V(\gamma)\big)\tfrac{\mathrm D}{\mathrm ds}\xi - g(\nabla V(\gamma),\xi)\dot \gamma\]
and taking the limit as $s \to 0$ in \eqref{eq:ii} allows to obtain also property (e).
\end{proof}

\section{On the $C^2$--regularity of the distance from the boundary of the potential well}\label{sec:reg}
In this last section we prove the $C^2$--regularity for the Jacobi distance from the boundary of the potential well in a neighborhood of a point $Q_0$ with a unique minimizer and such that $Q_0$ is not conjugate to $V^{-1}(E)$. Indeed we prove the following
\begin{teo}\label{teo:C2reg}
Let $d_V$ be the Jacobi distace from $V^{-1}(E)$, Assume that $Q_0 \in V^{-1}\big(\left]-\infty,E\right[\big)$ is such that there is a unique minimmizer that realizes $d_V(Q_0)$. Assume also that $Q_0$ is not conjugate to $V^{-1}(E)$. Then $d_V$ is of class $C^2$ in a neighborhood of $Q_0$.
\end{teo}
\begin{proof}
By Theorem \ref{thm:equiv}, it follows that $Q_0$ is a regular value of the exponential map. Since it is of class $C^2$ we see that any $Q$ sufficiently close to $Q_0$ is a regular value of the exponential map. Then, always by Theorem \ref{thm:equiv} we have that $Q$ is not conjugate to $V^{-1}(E)$.

Then, for any $Q$ sufficiently close to $Q_0$ there is a unique minimizer. Indeed suppose by contradiction there exists a sequence $Q_n$ of points with at least two minimizers $\gamma_n^1$ and $\gamma_n^2$. By Lemma 3.4 in \cite{MT}, they converge (with respect to the $H^1$--norm to the unique minimizer that realizes $d_V(Q_0)$. But this would be in contradiction with the fact that $Q_0$ is a regular value of the exponential map.

Finally, by Proposition 3.5 of \cite{MT}, for any $Q$ nearby $Q_0$ the gradient of $d_V$ at $Q$ is given by
\[
\nabla d_V(Q)=\frac{(E-V(Q))}{2d_V(Q)}\dot\gamma_Q(1),
\]
where $\gamma_Q$ is the unique minimizer between $V^{-1}(E)$ and $Q$, parameterized in the interval $[0,1]$. Then to prove the $C^2$ regularity of $d_V$ it suffices to prove the $C^1$--regularity of $\dot\gamma_Q(1)$. Since $\gamma_Q(1)=Q$, thanks to the invertibility of $\exp^{\perp}$ and its $C^1$--regularity, formula \eqref{eq:trastemp}
and Remark \ref {eq:regvel} allows to obtain the conclusion of the proof.

\end{proof}

\end{document}